\documentclass[11pt, oneside]{article}
\usepackage{geometry}                		
\usepackage[pagewise]{lineno}	
\usepackage[parfill]{parskip}    		
\usepackage{graphicx}				
\usepackage[utf8]{inputenc}
\usepackage[english]{babel}
\usepackage{mathrsfs,amsmath}
\usepackage{amsthm}								
\usepackage{amssymb}
\usepackage{amsmath}
\usepackage{amsfonts}
\usepackage{mathbbol}
\usepackage{biblatex}
\usepackage{xpatch}
\usepackage{subcaption}
\usepackage{hyperref}
\usepackage{csquotes}
\usepackage{lipsum}

\newtheorem{theorem}{Theorem}[section]
\makeatletter
\AtBeginDocument{\xpatchcmd{\@thm}{\thm@headpunct{.}}{\thm@headpunct{}}{}{}}
\makeatother
\newtheorem{lemma}[theorem]{Lemma}

\newtheorem{proposition}[theorem]{Proposition}

\newtheorem{corollary}[theorem]{Corollary}

\theoremstyle{definition}
\newtheorem{definition}{Definition}[section]

\theoremstyle{remark}
\newtheorem{remark}{Remark}[section]

\theoremstyle{definition}
\newtheorem{example}{Example}[section]

\theoremstyle{definition}

\date{}
\title{A standard form of incompressible surfaces in 3-dimensional handlebodies}

\author{Wei Lin, Fengchun Lei}

\begin{document}
\maketitle
\baselineskip 14pt

\setlength{\parindent}{1.5em}

\abstract{Applying Morse theory, we give a standard form for a class of surfaces which includes all the properly embedded incompressible surfaces in 3-dimensional handlebodies. We also give a necessary and sufficient condition to determine the incompressibility of such surfaces placed in our standard form. Our algorithm is practical. Several examples are given to test the algorithm.}

{\small{\bf Keywords}: incompressible surface, handlebody, Morse theory, general position


\maketitle

\section{Introduction}

The question of constructions for incompressible surface in genus two handlebody $H_2$ was first studied by Jaco \cite{7}, who showed there exists  a properly embedded, nonseparating, incompressible surface with arbitrarily large genus in $H_2$. Jaco then asked whether there exists a separating incompressible surface with arbitarily high genus in $H_2$. This question was answered in the affirmative independently by Eudave-Muñoz \cite{1}, Howards \cite{3}, and  Qiu \cite{8}. Nogueria and Segerman \cite{6} showed that for every compact surface with boundary, orientable or not, there is an incompressible embedding of the surface into $H_2$.\par

Several standard positions of incompressible surfaces in the handlebodies were given by different authors. Charitos, Papadoperakis, Tsapogas \cite{4} gave a canonical gluing method where the pieces are disks, so that every properly embedded incompressible surface in a handlebody can be constructed with. They also gave a simple sufficient condition which asserts that the result of the gluing process is incompressible. Let $F$ be a compact surface and let $I$ be the unit interval. Green \cite{2} gave a standard position for all 2-sided incompressible surfaces in the 3-manifold $F\times I$. It also supplies a simple sufficient condition for when 2-sided surfaces in this form are incompressible. And, since $F\times I$ is a handlebody when $F$ has boundary, the result applies to incompressible surfaces in handlebodies. Lei, Liu, Li, Vesnin \cite{5} gave a necessary and sufficient condition for a surface sum of two handlebodies along a connected surface to be a handlebody.\par

 We work on differentiable manifolds, including the handlebodies and the surfaces embedded in it. To construct the interested surfaces, we first associate the handlebody $V$ whose genus is larger than or equal to two, and the surface $S$ properly embedded in it with a height function $h$. We divide $V$ through a maximal essential disk system $\mathcal{D}$, and study $S$ as the union of the subsurfaces in each division. We say a properly embedded surface $S$ in $V$ is \emph{pseudo-essential}, if for each horizontal disk $D$ that intersects $S$ transversely, the intersection $S\cap D$ consists of arc(s) alone (for details, see $\S$\ref{C2}). We prove a result that completely determines whether a given psuedo-essential surface in standard form is incompressible or not.\par

We will show that every pseudo-essential surface can be constructed through a simple canonical gluing of disk components. We can also show that, the class of pseudo-essential surfaces includes the class of incompressible surfaces. Each disk component as a gluing piece of the surface is isotoped to a particular form, so that some disks will be considered interesting in the sense that they contain critical points of $h$. We will formulate the above mentioned in $\S$\ref{C3}.

Furthermore, when a pseudo-essential surface is placed in our standard form, we claim there exists an algorithm to tell if it is incompressible or not. And we give a complexity measurement $\mathcal{L}$ in $\S$\ref{C4}, so that $S$ is placed in \emph{standard position} when $\mathcal{L}$ is minimized. $S$ under such a placement can be considered well-behaved, in the sense that, if $S$ is incompressible, it possesses certain properties listed in Proposition \ref{no saddle}. We determine whether $S$ is incompressible or not by simply examining the combination of disks called \emph{polygons} in each component of $\overline{\{V\setminus \mathcal{D}\}-S}$, that can be altered as described in Definition \ref{polygon}. We let $S$ retract to a graph $\mathcal{G}_S$, whose vertices mark the critical points of $S$. Then we have the following statement as our main result: \par

\begin{theorem}\label{THM1}
Let $V$ be a handlebody whose genus is larger than or equal to two, $\mathcal{D}$ be a maximal essential disk system of $V$. Suppose $S\subset V$ is a connected pseudo-essential surface, that is placed in standard position with respect to $\mathcal{D}$. Then there is an algorithm of finite many steps to determine whether $S$ is incompressible or not. $S$ is compressible if and only if $\mathcal{G}_S$ is trivial, or there exists a set of polygons satisfy the following linear equation set: \par

\begin{equation}\label{equation1}
    \begin{cases}
      \frac{1}{2}(n_2+(-1)\cdot n_6+...+(2-m)\cdot n_{2m})=1\\
      n_2+(2)\cdot n_4+...+(m)\cdot n_{2m}=|\mathcal{E}_{\mathcal{D}}|=|\mathcal{E}_S|\\
      n_2\leq \mathcal{N}_c\cdot|\mathcal{P}_2|
    \end{cases}
\end{equation}

where $m$, $\mathcal{N}_c$, $|\mathcal{E}_{\mathcal{D}}|$, $|\mathcal{E}_S|$ and $|\mathcal{P}_2|$ are positive integers determined by $S$, and $n_{2m}$ is the amount of polygons with $2m$ edges.\par

\end{theorem}

See $\S$\ref{Algorithm} for the algorithm mentioned in the above theorem. We notice that the above equation set has at most finitely many positive integer solutions.

If $S$ is properly embedded in a thickened surface $F\times [0,1]$, we have the following corollary that coincides with a probably well-known result, stated as Proposition 3.1 in \cite{9}:\par

\begin{corollary}\label{COR1}
Let $V$ be a handlebody, $\mathcal{D}$ be a maximal essential disk system of $V$. Suppose $S$ is a pseudo-essential surface that can be isotoped so that $S \subset \partial V \times [0,1] \subset V$. When $S$ is placed in standard position with respect to $\mathcal{D}$, it is incompressible if and only if $\mathcal{G}_S$ is nontrivial, and $S$ is $\partial$-parallel to $\partial V$ .\par
\end{corollary}

In $\S$\ref{C2} we give some basic notations applied in this paper. In $\S$\ref{C3} we give a construction of pseudo-essential surfaces. In $\S$\ref{C4} we give a \emph{standard position} for pseudo-essential surfaces. In $\S$\ref{C5} we give a \emph{normal position} to regulate the compressing disks. We claim that any compressing disks, if there exists any, can be put or divided into a union of disk(s) in such a position. In $\S$\ref{C6} we demonstrate the application of our algorithm through some examples.\par

\section{Preliminary Background} \label{C2}

Let $M$ be a 3-dimensional orientable differentiable manifold with or without boundary. Let $S\subset M$ be a properly embedded surface, i.e., the interior Int$(S)$ and the boundary $\partial S$ of $S$ satisfy the inclusions Int$(S)\subset$ Int$(M)$ and $\partial S \subset \partial M$. A \emph{compressing disk} for $S\subset M$ is an embedded disk $D\subset M$ such that $\partial D \subset S$, and Int$(D)\subset M\setminus S$ and $\partial D$ is an essential loop in $S$, i.e. the map $\partial D \to S$ induces an injection $\pi_1(\partial D) \to \pi_1(S)$. A properly embedded surface $S\subset M$ is \emph{incompressible} if there are no compressible disks for $S$ and no component of $S$ is a sphere that bounds a ball. If $S\subset M$ is connected and 2-sided (i.e., if $S$ is orientable) then $S$ is incompressible if and only if the induced map $\pi_1(S) \to \pi_1(M)$ is injective and $S$ is not a sphere that bounds a ball. Incompressible surface is an important subject of research in the theory of 3-manifolds. \par

Throughout the paper, we assume all the manifolds involved are differentiable. Let $V$ be a 3-dimensional handlebody, generally assumed to have genus larger than or equal to two. We place $V$ vertically, choose a Morse function $h: \mathbb{R}^3\to \mathbb{R}$, $\forall u=(x,y,z)\in \mathbb{R}^3 $, $h(u)=z$, and a \emph{maximal essential disk system} $\mathcal{D}=\{D_1,\cdots,D_{3g-3}\}$, such that each of $D_1$, $D_2$, $\{D_{3},D_{4}\}$, $\cdots, \{D_{3g-6},D_{3g-5}\}$, $D_{3g-4}$, and $D_{3g-3}$ is a preimage of a regular value of $h$ at different heights, shown as in Figure \ref{f1} below. A disk $D$ in $V$ is called a \emph{horizontal disk} if it is a component of the preimage of a regular value of $h$. We give the definition for the \emph{pseudo-essential surfaces} in handlebodies as the class of subjects we study in this paper. \par

\begin{figure}[h!]
\centering
  \includegraphics[width=9.5cm]{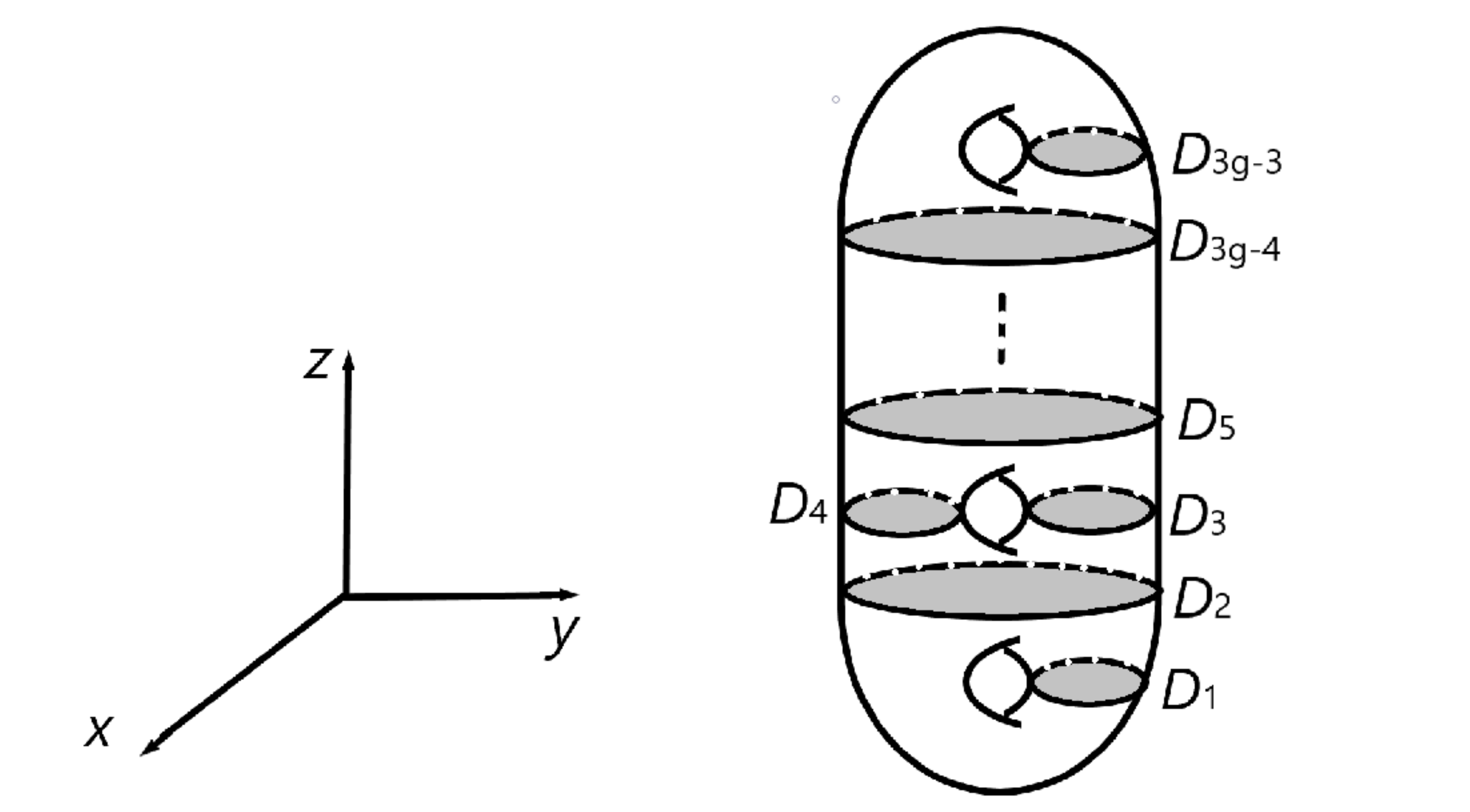}
  \caption{}
\label{f1}
\end{figure}

\begin{definition}\label{pseudo}
A properly embedded surface $S$ in $V$ is \emph{pseudo-essential} (with respect to a maximal essential disk system $\mathcal{D}$ and the height function $h$), if for each horizontal disk $D$ that intersects $S$ transversely, the intersection $S\cap D$ consists of arc(s) alone. \par
\end{definition}

Let $C$ be the spine (a 3-valent graph) that $V$ retracts to, so that a disk in $\mathcal{D}$ retracts to a point in the interior of an edge in $C$, as in Figure \ref{f2}(a). For a surface $S$ properly embedded in $V$, we isotope $S$ in $V$, while keeping $\partial S\subset \partial V$, so that $|S\cap C|$ reaches the minimum, and we fix the points of $S\cap C$ prior to any isotopies of $S$ for the rest of this paper. We use the notation $h_{M_S}=h|_{M_S}:M_S\to \mathbb{R}$ to denote the Morse function on the submanifold $M_S$. Following from Morse theory, we can require the critical heights of $h_{\partial V}$, $h_S$, and $h_{\partial S}$ are all regular and distinct. We may also require these values to be distinct from the discrete values $h_{\mathcal{D}}$. And we say that $S$ is in a \emph{general position} with $\mathcal{D}$ and $C$, if $S$ meets all the conditions above. \par

$\mathcal{D}$ cuts $V$ into $2g-2$ 3-balls, we consider the closure of them, $\{P_k,1\leq k\leq 2g-2\}$, each is called a {pant-shaped} ball, or simply, a ps-ball, see (b) and (c) in Figure \ref{f2}. A ps-ball $P_i$ has 3 horizontal disks in $\partial P_k$ from copies of $\mathcal{D}$, see Figure \ref{f2}(c).\par

\begin{figure}[h!]
\centering
  \includegraphics[width=12cm]{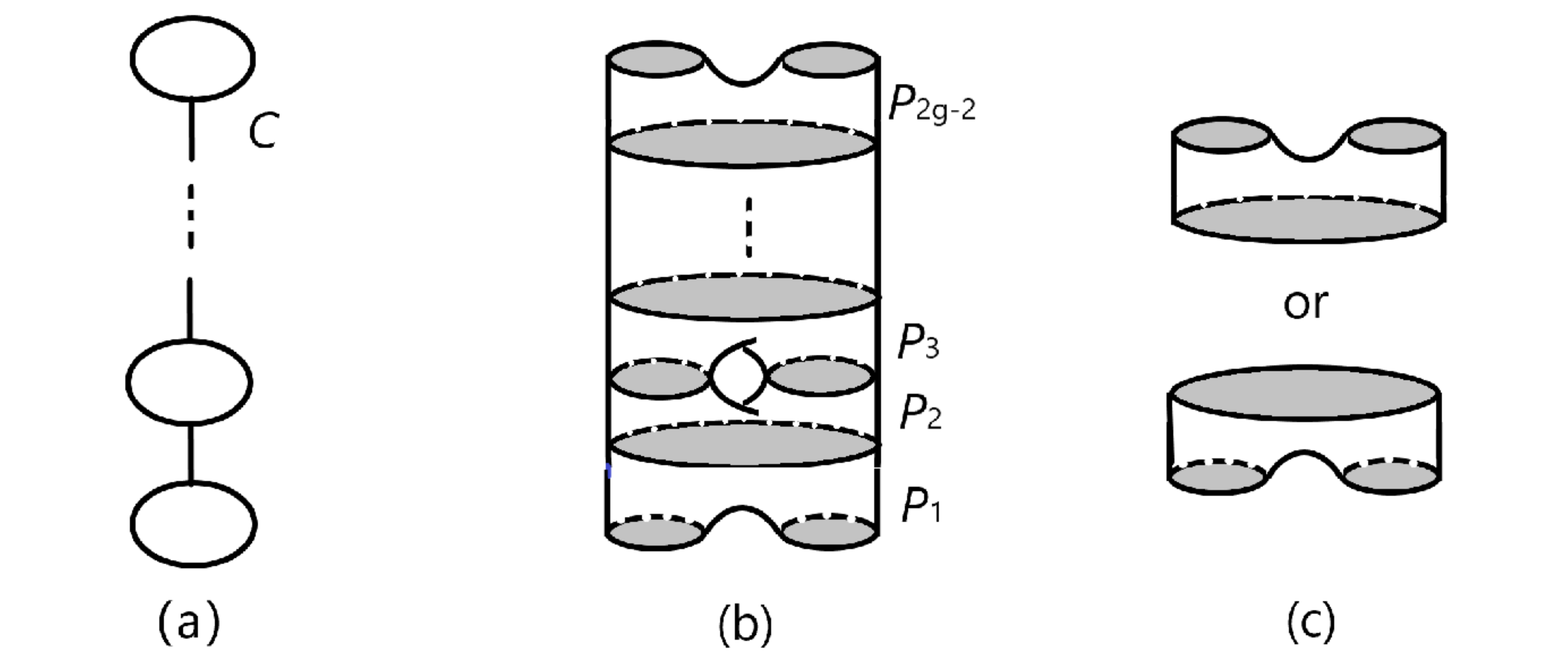}
  \caption{} \label{f2}
\end{figure}

Let $s_1,s_2,\cdots,s_K$ be all the critical heights of $h_S$ and $h_{\partial S}$ with $s_1<s_2<\cdots<s_K$. We can take an enough small $\varepsilon_i>0$, so that $s_i$ is the only critical height of all the critical heights of $h_S$ and $h_{\partial S}$ in $I_i=[s_i-\varepsilon_i,s_i+\varepsilon_i]$. Let $W_i$ denote the $I$-product structures, so that each $W_i =h^{-1}(I_i)=E_i\times I_i$, and that $E_i\times \{z_i\}, z_i\in I_i$ is a horizontal disk. Additionally, we can adjust each $W_i$, so that each component of the complements of $\{\cup W_i\}$ (the union of all such $W_i$'s) in $V$ is one of the following two cases:\par

(I) A closed ball that has three disk intersections with the boundaries of some $W_i$'s. and such a component is called a \emph{short ps-ball}, denoted as $X_j$, $X_j \in \overline{V\setminus \{\cup W_i\}}$, $1\leq j \leq 2g$; \par

(II) A closed ball which contains the highest or lowest critical height, that has only one disk intersection with the boundary of one of the $W_i$'s. \par

We may require (II) to have empty intersections with $S$, and it is considered as the trivial case. It follows from Morse theory that $S$ can be chosen as a representative of its isotopy class to satisfy the above conditions. \par

\section{The Construction of pseudo-essential Surfaces} \label{C3}

We now give the definition for the gluing disk components mentioned in the introduction.

\begin{definition}\label{2ndisk}
Let $S_{2n}^{i}$ denote a disk in a product $W_i$ whose boundary $\partial S_{2n}^{i}$ consists of $n$ edges in $E_i\times \partial I_i$, and $n$ edges in $\partial E_i\times I_i$, where $n= 1, 2, 3$, or $4, 6, ... , 2k$, and k is a positive integer.  Similarly, we use the notation $S_{2n}^{j}$ to denote a disk in a short ps-ball $X_j$, where $n=1, 2, 3$. Such a disk is called a \emph{2n-disk} if it subjects to the following restrictions:\par

(1) For each horizontal disk that intersects $S^i_{2n}$ or $S_{2n}^{j}$ transversely, the intersection consists of properly embedded arcs.

(2) Each $2$-disk contains a critical point of $h_{\partial S}$ on $S\cap \partial V$. See Figure \ref{f3}(a).

(3) Each $4$-disk in $W_i$ contains no critical point, and each $4$-disk in $X_j$ may contain a critical point of $h_{\partial S}$ in each arc on $S\cap \partial V$. See Figure \ref{f3}(b).\par

(4) Each $6$-disk contains a critical point of $h_{\partial S}$ on $S\cap \partial V$. See Figure \ref{f3}(c).\par

(5) Each $2N$-disk (where $N$ is an even integer, $N\geq 4$) is a non-degenerate saddle that contains a critical point $c_{i}$ of $h_S$ in its interior. $c_{i}$ is contained in a horizontal disk $E_i\times \{s_i\}$, so that $E_i\times \{s_i\}\cap S_{2N}$ is the only non-transverse intersection of $E_i\times \{z_i\} \cap S_{2N}$ for all $z_i\in I_i$. Moreover, $E_i\times \{s_i\}\cap S_{2N}$ is a union of $N$ arcs intersecting each other(s) at  $c_{i}$. See Figure \ref{f3}(d) for an example of a $12$-disk.\par

\end{definition}

\begin{figure}[h!]
\centering
  \includegraphics[width=16cm]{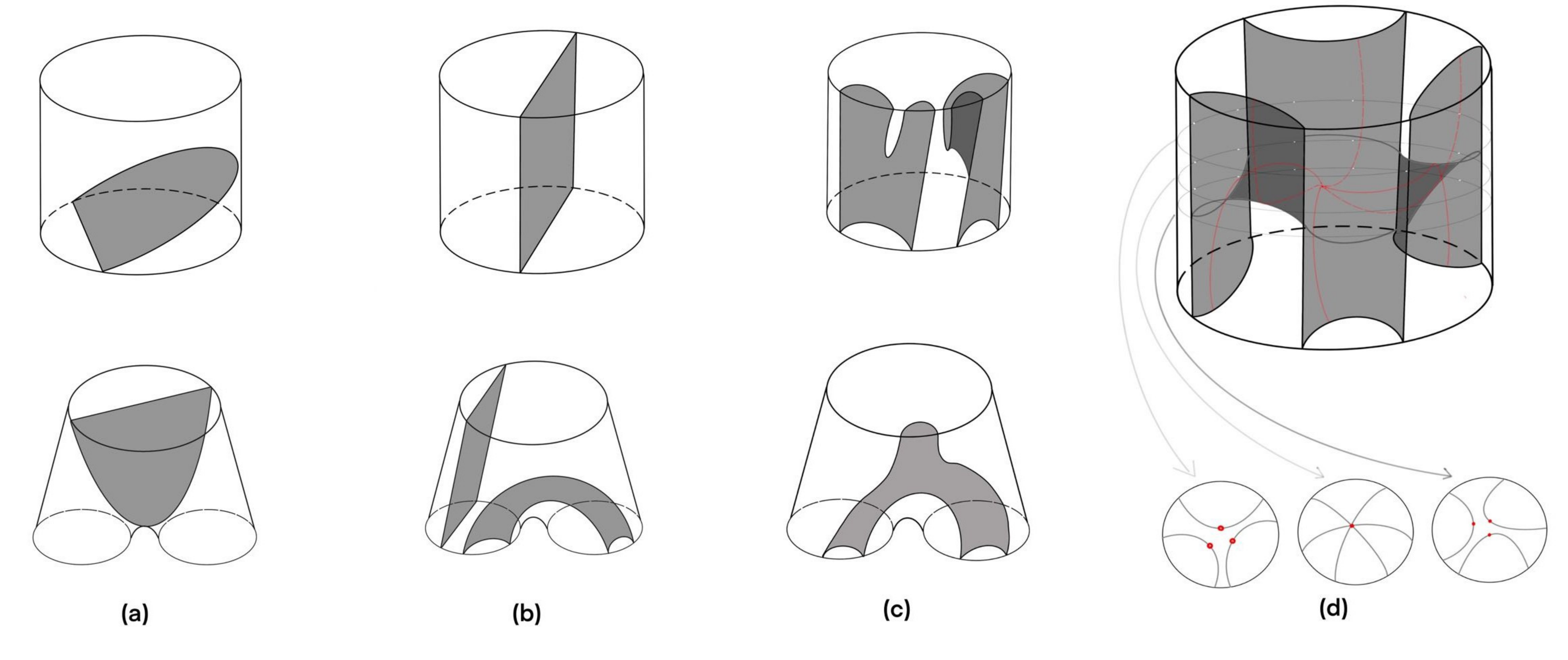}
  \caption{} \label{f3}
\end{figure}

We say a $2n$-disk is \emph{trivial} if it contains no critical points, otherwise we say it is \emph{nontrivial}. We now show that given a maximal disk system $\mathcal{D}$ and a height function $h$, any properly embedded incompressible surface $S\subset V$ is pseudo-essential.

\begin{proposition}\label{essential}
A pseudo-essential surface can be isotoped so that each of its disk component is as described in Definition \ref{2ndisk}. Let $S$ be an incompressible surface properly embedded in $V$, then $S$ can be isotoped so it is pseudo-essential.\par
\end{proposition}

\begin{proof}

As in the above section, we place the handlebody $V$ vertically. Let $C$ be the spine (a 3-valent graph) that $V$ retracts to. We pick a maximal essential disk system $\mathcal{D}$, so that a disk in $\mathcal{D}$ retracts to a point in the interior of an edge in $C$. Then we isotope $S$ so that $|S\cap C|$ reaches minimum and $S$ is in a general position with $\mathcal{D}$ and $C$. We consider a connected component $S^i_{2n} \subset S \cap W_i$, or $S^j_{2n} \subset S\cap X_j$. Since $S$ is incompressible in $V$, $S^i_{2n}$ (or $S^j_{2n}$)  must be an incompressible surface properly embedded in the 3-ball $W_i$ or $X_j$, which implies $S^i_{2n}$ (or $S^j_{2n}$) is a disk. We assume $S^i_{2n}$ (or $S^j_{2n}$) has $2n$ edges, and all isotopies mentioned below are perform in either $W_i$ or $X_j$. \par

When $n=1$, obviously we can isotope $S^i_{2n}$ or $S^j_{2n}$ so that it only contains a critical point of $h_{\partial S}$ on $S\cap \partial V$, and it is a $2$-disk.\par

When $n=2$, in $W_i$ we can isotope the arcs of $S^i_{2n}\cap \partial V$, so that it is ``vertical'', i.e., it intersects each horizontal disk $E_i\times \{z_i\}$ in single point, and it consists of no critical points of $h_{\partial S}$. In $X_j$ we can either perform the same isotopy, or when the edges of $S^j_{2n}$ intersect two disks in $\mathcal{D}$, we isotope $\partial S^j_{2n}$ so that $S^j_{2n} \cap \partial V$ contains only one critical point of $h_{\partial S}$ in each edge of its boundary in $\partial V$. In either case, the interior of  $S^i_{2n}$ or $S^j_{2n}$ consists of no critical point of $h_S$, and we can isotope $S^i_{2n}$ or $S^j_{2n}$ so that it is a $4$-disk. \par

When $n=3$, we consider the edges of $\partial S^i_{2n}$ (or $S^j_{2n}$). We can isotope $S^i_{2n}$ such that each of two of its edge in $\partial S^i_{2n}\cap \partial E_i \times I_i$, $A_1, A_2 \subset S^i_{2n} \cap \{\partial E_i\times I_i\}$ intersects each horizontal disk $E_i\times \{z_i\}$ in single point, and the edge $A_3$ is the only edge that contains the critical point of $h_{\partial S}$. $S^i_{2n}$ (or $S^j_{2n}$) is not regarded as a non-degenerate saddle, and the interior of it can be altered so it contains no critical point of $h_S$. Therefore, we can isotope $S^i_{2n}$ such that it is a $6$-disk.\par

When $n=N$, $N\geq 4$, and if we can isotope $S^i_{2N}$ such that each edge of $\partial S^i_{2N}$, $A_t\subset S^i_{2N} \cap \{\partial E_i\times I_i\}$ intersects each horizontal disk $E_i\times \{z_i\}$ in single point. Then $\partial S^i_{2N}$ does not contain a critical point. We can isotope $S^i_{2N}$ so that it is a non-degenerate saddle, and the only non-transverse intersection between any of the horizontal disks with $S\cap W_i$ is in $E_i\times \{s_i\}$, so that $E_i\times \{s_i\}\cap S_{2N}$ is a union of $N$ arcs intersecting each other(s) at the critical point of $h_S$. See Figure \ref{f3}(d). Therefore we have shown that $S^i_{2N}$ is isotoped to a $2N$-disk. Notice we can always adjust $S$ so that  a $2N$-disk, $N\geq 4$ is away from any $X_j$.\par

When $n=N$, $N\geq 4$, and there exists $m$ edges in $\partial V$ that intersects a horizontal disk of $E_i\times \partial I_i$ twice. As in the above arguments, we can deform $\partial S^i_{2N}$ so that each of these edges contains a critical point of $h_{\partial S}$. Then we divide $I_i$ into the union of $m+1$ subintervals, $I_i=I_i^1 \cup I_i^2\cup ... \cup I_i^{m+1}$, and regard $W_i$ as a union of sub-products, $W_i=E_i\times I^1_i\cup E_i\times I^2_i\cup ... \cup E_i\times I_i^{m+1}$. We can isotope $S^i_{2n}$ so that each critical point of $h_{\partial S}$ is the only critical point of $h_{\partial S}$ in one of the sub-products. By previous arguments, we can have each component of $S^i_{2n}\cap E_i\times I^s_i$, $1\leq s \leq  m+1$ that contains a critical point of $h_{\partial S}$ isotoped to $6$-disk. And we can isotope the other component of $S^i_{2n}\cap E_i\times I^s_i$ in each sub-product to a $4$-disk, with an exception that one of them is possibly a $2n'$-disk, $n'=n-2m\geq 8$. By the previous argument, we can deform this $2n'$-disk in the sub-product so that it is a non-degenerate saddle. See Figure \ref{f4} for an example: We separate $W_i$ into a union of two products, and we isotope the disk so that each product contains only one critical point of $h_S$ or $h_{\partial S}$. On the right, the upper product contains a $4$-disk and a $6$-disk, and the lower product contains a $8$-disk. Now that each product contains only one critical point of $h_{\partial S}$ or $h_S$, we can change the subscripts of them, so they are in accord with the other $W_i$'s. As such, each sub-product contains only disk(s) satisfying the conditions in Definition \ref{2ndisk}.\par
\begin{figure}[h!]
            \centering
            \includegraphics[width=8cm]{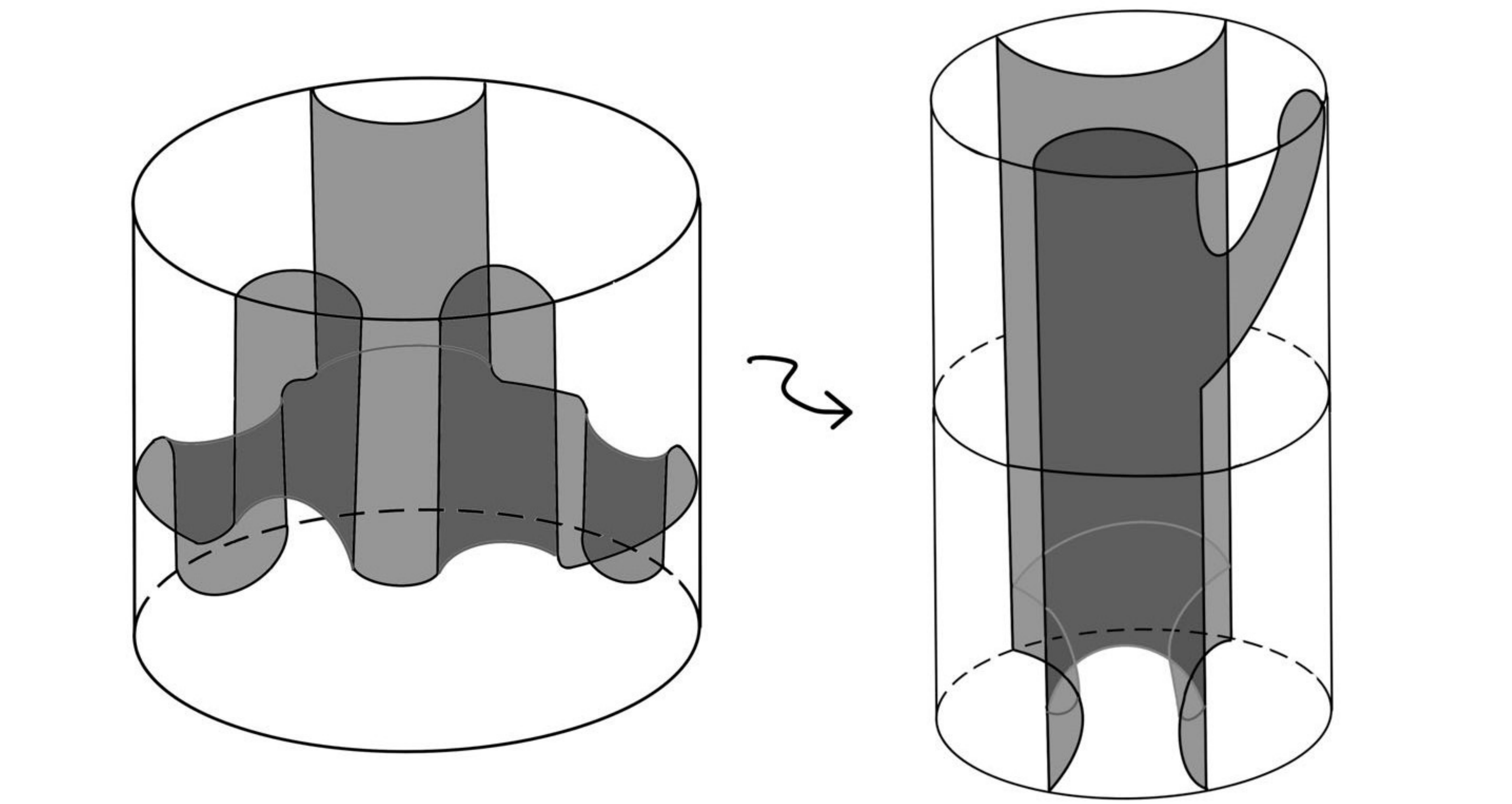}
            \caption{}
            \label{f4}
\end{figure}
Notice that, we have also proved $S$ can be isotoped so it is pseudo-essential with respect to $\mathcal{D}$ and $h$. Since for each horizontal disk that intersects a $2n$-disk transversely, the intersection consists of arcs alone.
\end{proof}

\section{Placing Surfaces in the Standard Position}\label{C4}

Consider a pseudo-essential surface $S$. Let $\eta(C)$ be a regular open neighborhood of $C$, then $\overline{V-\eta(C)\setminus \mathcal{D}}=\cup_{k=1}^{2g-2}PT_k$ is a union of ``thickened pair of pants''. Moreover, $PT_k=p_k \times [0,1]$, and $\partial PT_k$ is the union of $p_k \times \{0\}$, $ p_k \times \{1\}$. Here, $p_k$ is a ``pair of pants'', and $p_k \times \{0\}$ is regarded as the outer one. Let $A^x_k\subset \mathcal{D}\setminus \eta(C), 1\leq x \leq 3$ denote the annulus on $\partial PT_k$ that intersects $\mathcal{D}$. The conditions of general position require the number $|S\cap C|=|S\cap \eta(C)|$ to reach minimum, and each of the three annuli $A^x_k$ on the boundary of $PT_k$ contains transverse intersections of $A^x_k \cap S$, which are properly embedded arcs whose boundaries are on the outer circle of $\partial A^x_k$. \par

We say two arcs of $S\cap A^x_k$ has an \emph{inner-outer relation} if they co-bound a sub-disk in $A^x_k$ with $\partial A^x_k$. The arc that is closest to the outer boundary of the annulus is called an \emph{outermost arc}, and the arc that is closest to the inner boundary of the annulus is called an \emph{innermost arc}. The arcs of $S\cap A^x_k$ in each annulus can be regarded as a family of (not necessarily disjoint) sets $\cup \mathcal{A}^x_{km}$, $m\geq 1$, so that each $\mathcal{A}^x_{km}$ either consists of only one arc, or any two arcs in $\mathcal{A}^x_{km}$ admits an inner-outer relation. Moreover, each $\mathcal{A}^x_{km}$ is maximal in the sense that all arcs that has an inner-outer relation with the outermost arc, including the outermost arc, are all contained in it. For example, in Figure \ref{f5}(a), the annulus consists of two such families, each contains the red arc. And the arcs that are the same color has an inner-outer relation with each other. In particular, the red arc is an innermost arc that has an inner-outer relation with all other arcs. \par

\begin{figure}[h!]
            \centering
            \includegraphics[width=15cm]{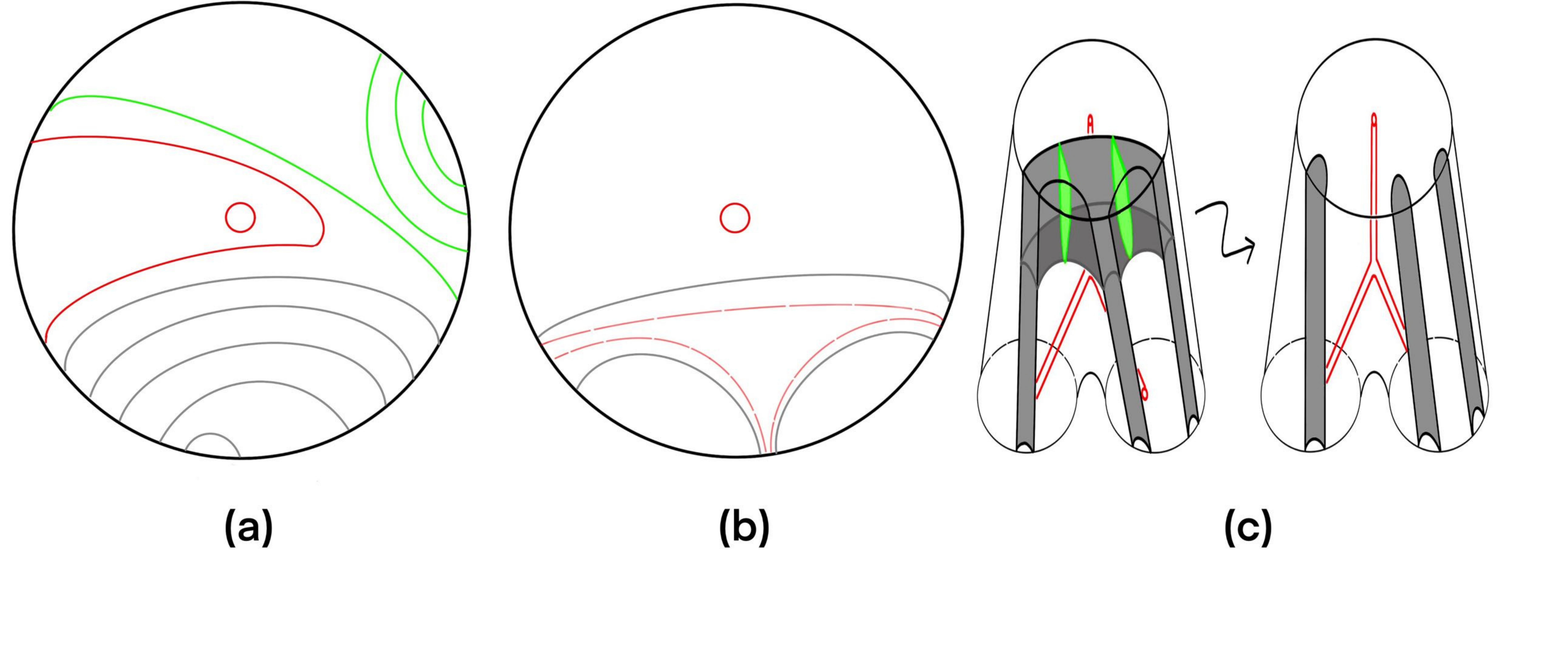}
            \caption{}
            \label{f5}
\end{figure}

As such, we call each $\mathcal{A}^x_{km}$ a \emph{maximal arc set}, and we can define a counting function $f$ whose domain is such families of arcs: $\{\mathcal{A}^x_{km},1\leq x \leq 3, 1\leq k \leq 2g-2, m\geq 1\}$, so that $f(\mathcal{A}^x_{km})=|\mathcal{A}^x_{km}|$. Let $L$ be the maximal value of the counting function $f$. For each $2n$-disk $S^i_{2n}$ that has nonempty intersections with an annulus $A$ parallel to some $A^x_k$, we can consider the arcs of $S^i_{2n}\cap A$. If $S$ intersects $A$ transversely, then the definitions above stil applies on $A$ and $S^i_{2n}\cap A$. If $S^i_{2n}\cap A$ consist of two arcs that admit an inner-outer relation, then we call $S^i_{2n}$ a \emph{movable saddle}.  We denote the set of all movable saddles as $\mathcal{D}_{ms}$. It is obvious we can isotope $S$ so that all movable saddles are contained in $W_i\cap PT_k$ for some $i$'s and $k$'s, i.e., $\mathcal{D}_{ms}\cap \eta(C)=\emptyset$. We can also assume without loss of generality, that $S$ is still in standard position after such isotopies.\par

\begin{definition}\label{DEF}
We say a pseudo-essential surface $S$ is in \emph{standard position} (with respect to $\mathcal{D}$ and $h$) if $S$ satisfies all the conditions in Definition \ref{2ndisk}, and the $(L+1)$-tuple

\begin{center}
 $\mathcal{L}=(|S\cap \mathcal{D}|,|f^{-1}(L)|,|f^{-1}(L-1)|,...,|f^{-1}(2)|, |\mathcal{D}_{ms}|)$ \par
\end{center}

is minimized in the lexicographic order, where each $|f^{-1}(l)|, 2\leq l \leq L$, is the amount of maximal arc sets $\mathcal{A}^x_{km}, 1\leq x \leq 3, 1\leq k \leq 2g-2, m\geq 1$, so that $f(\mathcal{A}^x_{km})=l$.

\end{definition}

We study how certain isotopies will affect $\mathcal{L}$. Let $S_{2n}\subset PT_k$ be a movable saddle, we say $S_{2n}$ is an \emph{outermost movable saddle} with respect to an annulus $A^x_k$, if there does not exists any arcs of intersection on the sub-disk corresponding to the $\frac{n}{2}$ arcs of $S_{2n}\cap A^x_k$. For example, in Figure \ref{f5}(b), each of the red dotted arcs is on the sub-disk corresponding to the 3 arcs of some 12-disk $S_{12}\cap A^x_k$.

\par

\begin{lemma}\label{arc sum}
Given an outermost movable saddle $S_{2n}\subset PT_k$ with respect to an annulus $A^x_k$,  there is an isotopy which induces $(\frac{n}{2}-1)$ $\partial$-compressions on the subsurface $S\cap PT_K$ and decreases $\mathcal{L}$.

\end{lemma}

\begin{proof}
Given an outermost movable saddle $S_{2n}\subset PT_k$ with respect to an annulus $A^x_k$, we can fix the boundary of $S_{2n}$ on the ``pair of pants'', i.e., $p_k \times \{0\}$. Then we isotope the interior of $S_{2n}$, so that the critical point contained in it is moved outside of $PT_k$, and $A^x_k$ intersects with the $(\frac{n}{2}-1)$ sub-surfaces in $PT_k$ in $(\frac{n}{2}-1)$ arcs after such an isotopy. This isotopy will simultaneously induce $(\frac{n}{2}-1)$ $\partial$-compressions on the compressed surface $S_{2n}$ with the boundary $\partial PT_k$. It is easy to check that each of the induced $\partial$-compressions will reduce $\mathcal{L}$, even if $S$ has other intersections with the corresponding annulus. We notice that the number $|S\cap \mathcal{D}|$ does not change after such $\partial$-compressions. See the isotopy shown in Figure \ref{f5}(c) for an example, where $n=6$.
\end{proof}

We induce a graph related to $2n$-disks. Let each $4$-disk retract to an arc that connects two horizontal disks; Let each $6$-disk retract to 3 arcs that intersect at a $3$-valent vertex as the critical point of $h_{\partial S}$ in $\partial S$. And let each $2n$-disk, $n=4, 6$, ... ,$2k$, retract to $n$ arcs that intersect a $n$-valent vertex as the critical point of $h_S$ in its interior. Then we have $S$ retracted to a graph $\mathcal{G}_S$, with each vertex marking a critical point of $h_S$ or $h_{\partial S}$. For convenience, we assume an edge $e$ of $\mathcal{G}_S$ may contains vertices.\par

\begin{lemma}\label{L}

Let $S$ be an incompressible surface in standard position, $P_k$ be a ps-ball, then there exists no such pair of edges, $e_1$, $e_2 \subset \mathcal{G}_S \cap P_k$, that they intersect at a vertex $v$, and that they intersect some annulus $A_k^x\subset \partial P_k$ in two arcs, $\alpha_1$ and $\alpha_2$, which have an inner-outer relation.

\end{lemma}

\begin{figure}[h!]

            \centering
            \includegraphics[width=10cm]{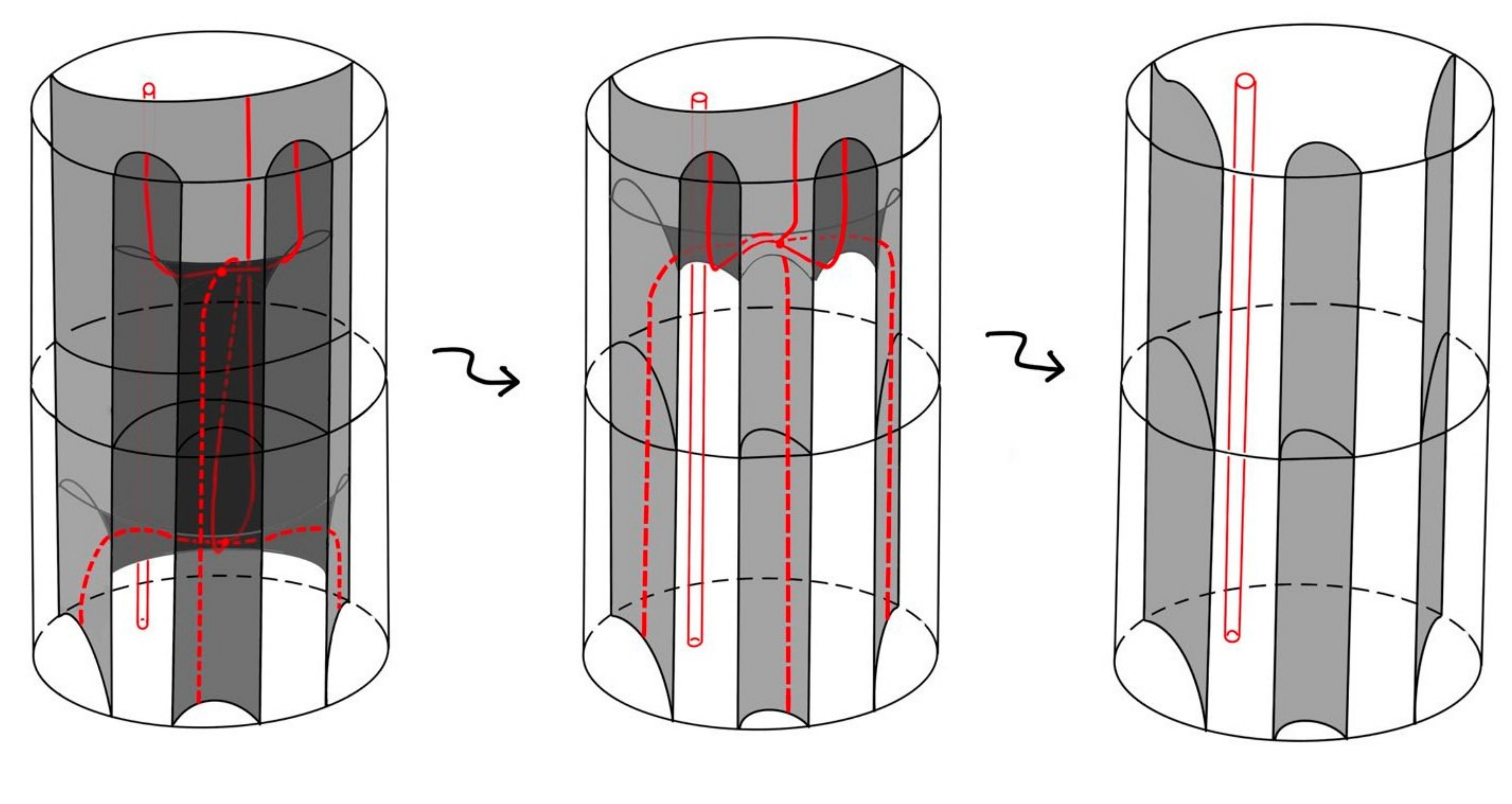}
            \caption{}
            \label{f6}

\end{figure}

\begin{proof}

We first claim that neither $e_1$ nor $e_2$ consists of a 3-valent vertex $v_3$ that marks a critical point of a $6$-disk $S_{6}$, which intersects $A_k^x$ in two arcs. Because, if either $e_1$ or $e_2$ consists such a vertex, then obviously we can isotope $S_{6}$ while moving $v_3$ out of $P_k$  to reduce $|S\cap \mathcal{D}|$, and therefore we can reduce $\mathcal{L}$. We notice that the condition $S$ being incompressible is needed (but not necessary), since otherwise there could be a loop of $\mathcal{G}_S$ that consists of $e_1$ and $e_2$, and $|S\cap \mathcal{D}|$ would not be necessarily reduced.\par

If both $e_1$ and $e_2$ consist of a single vertex $v$, then we have an outermost movable saddle, whose critical point is $v$.  Therefore, by Lemma \ref{arc sum}, we can reduce $\mathcal{L}$. Now we can assume without of loss of generality that, other than $v$, $e_1$ contains only one more vertex $v'$, and $e_2$ does not contain any vertices. Then since $\mathcal{G}_S$ consists no loops, we can reduce the number of vertices by one through the isotopy shown in Figure \ref{f6}. And then we have an outermost movable saddle, so we can reduce $\mathcal{L}$. We notice Figure \ref{f6} only illustrates the situation where both $2n$-disk are away from $\eta(C)$. If one of them intersects $\eta(C)$, we can also have a similar isotopy to reduce $\mathcal{L}$. The situation of either $e_1$ or $e_2$ consists of more than two vertices is similar, since we can reduce the vertex number through the above-mentioned isotopies.
\end{proof}

\begin{proposition}\label{no saddle}
Let $S$ be an incompressible surface properly embedded in standard position with respect to $\mathcal{D}$. Then $S$ has the following properties:\par

1. The subgraphs $\mathcal{G}_S\setminus \mathcal{D}$ are acyclic;\par
2. $S$ contains no movable saddles;\par
3. A $2N$-disk $S_{2N}$, $N\geq 5$, must intersect $C$; \par
4. Any $2$-disk is contained in some short ps-ball $X_j$;\par
5. A 3-valent vertex in a component of $\mathcal{G}_S\setminus \mathcal{D}$ meets edges that intersect with all three annuli on a ps-ball $P_k$. \par

\end{proposition}

\begin{proof}

Property 1 is obvious, since otherwise it would contradict the condition $S$ is incompressible.\par

Assume $S$ is in standard position. Let $P_k$ be a ps-ball, $S_{2n}$ be an movable saddle which consists of a vertex $v$ of $\mathcal{G}_S$ in its interior. Because $\mathcal{G}_S\cap P_k$ consists of no loops, by Lemma \ref{L}, any two edges $e_1$ and $e_2$ which intersect at a vertex $v$ must also intersect each annulus $A_k^x \subset \partial P_k$ in two arcs that do not has an inner-outer relation. Therefore we can isotope $S\cap P_k$ to reduce the critical point number, and reduce $|\mathcal{D}_{ms}|$, see Figure \ref{f7}(a). We have shown that property 2 is true. \par

\begin{figure}[h!]
            \centering
            \includegraphics[width=12cm]{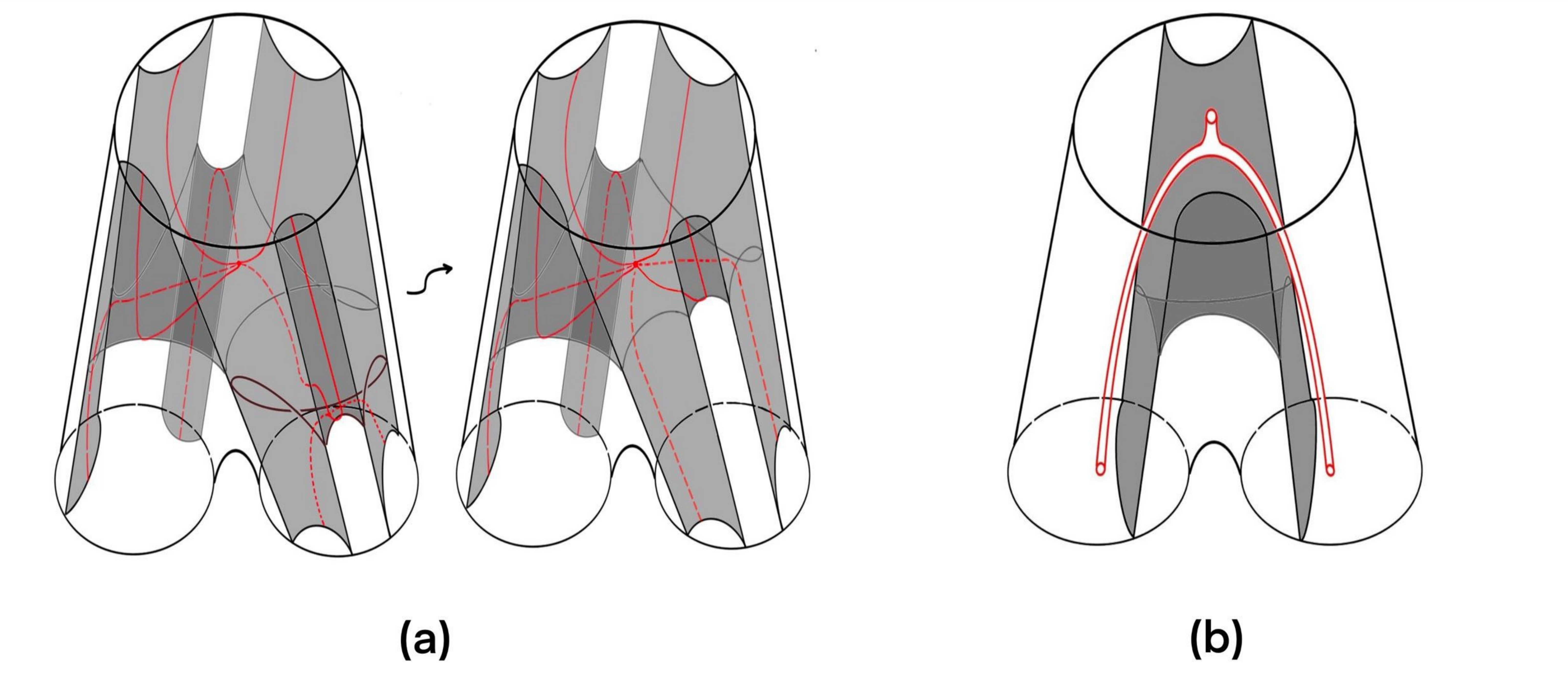}
            \caption{}
            \label{f7}
\end{figure}

To see property 3, we assume a $2n$-disk $S_{2n}$, $n=4, 6$, ... , $2k$, consists of a critical point in its interior, and $S_{2n}$ does not intersect $C$. Then $S_{2n}$ must be a movable saddle, unless it is a $8$-disk shown as in Figure \ref{f7}(b). \par

To see property 4, we assume by contradiction that, if there exists a $2$-disk in some $W_i$ which does not intersect $\mathcal{D}$, we can always isotope $S$ so the $2$-disk is removed form $W_i$. Then, if the $2$-disk were in a product $W_i$ that intersects with $\mathcal{D}$,  $|S\cap \mathcal{D}|$ would be reduced by the isotopy that removes the $2$-disk from $W_i$. Therefore the $2$-disk must be embedded in some short ps-ball $X_j$.\par

To see property 5, if two edges of $\mathcal{G}_S\setminus \mathcal{D}$ that meet at a 3-valent vertex $v_3$ in $P_k$ both intersect with an annulus $A_k^x$, then apparently we can reduce $|S\cap \mathcal{D}|$ through an isotopy of $S$ while moving $v_3$ out of $P_k$.
\end{proof}

As a result, it is reasonable to give the following definition: 

\begin{definition}
Let $S$ be a pseudo-essential surface properly embedded in $V$. If $S$ is in standard position, and there exists a component of $\mathcal{G}_S\setminus \mathcal{D}$ that consists of a cyclic graph, then we say $\mathcal{G}_S$ is \emph{trivial}. Otherwise, we say $\mathcal{G}_S$ is \emph{nontrivial}. Additionally, if a saddle contained in $S$ is not a movable saddle, we call it an \emph{essential saddle}.
\end{definition}

\section{Compressing disks in Normal Position} \label{C5}

Due to Theorem \ref{essential}, any incompressible surface $S$ properly embedded in $V$ can be put in standard position, and according to Proposition \ref{no saddle}, when $S$ is in standard position, it has no movable saddles. But the contrary is not necessarily true. We want to study the necessary and sufficient conditions for a compressing disk to exist when $S$ is in standard position, so that we can determine incompressibility of $S$.  \par

We say a properly embedded arc $\beta$ in $S$ is a \emph{horizontal arc}, if for a horizontal disk $D$, $\beta \subset S\cap D$; We say $\beta$ is an \emph{essential arc} provided $\beta$ cannot be deformed (fixing $\partial b$) in $S$, into $\partial S$; Otherwise we say it is an \emph{inessential arc}. \par

\begin{definition} \label{polygon}
A sub-disk $p_{D_c}$ in $D_c$ with $p_{D_c} \cap \{S_k\cup \mathcal{D}\}=\partial p_{D_c}$ is called a \emph{polygon} if it satisfies the following conditions: \par
1. The boundary $\partial p_{D_c}$ of each polygon $p_{D_c}$ in $D_c$ intersects each essential horizontal arc $\beta$ at most once, and $D_c$ does not transversely intersect any inessential horizontal arcs in one point; \par
2. For each nontrivial $2n$-disk $S_{2n}$ that intersects $p_{D_c}$, where $n=2, 3, 4, 6$, ... ,$2k$, $D_c \cap S_{2n}$ is a properly embedded arc in $\partial D_c$ which consists of a single critical point of $h|_{\partial D_c}$ that is in a small neighborhood of the critical point in $S_{2n}$; \par
3. A bi-gon $p_2$ has to intersect an essential saddle; \par
4. A polygon $p_{2m}$ is a non-degenerate saddle disk if and only if there exists 4 edges (arranged in order, if we orient $\partial p_{2m}$) of $p_{2m}\cap \mathcal{D}$ that intersect with disks of different heights in $\mathcal{D}$ alternatingly.  \par

\end{definition}

\begin{definition}\label{definition of disk}
Let $S\subset V$ be a compressible surface in standard position that contains no movable saddles. We use $N_{m}$ to denote the total number of $2m$-gons, for all $m\geq 4$. If $S$ has no movable saddles, we say a disk $D_c$ properly embedded in the complement of $S$ is in \emph{normal position}, if $D_c$ intersects each $P_k$ in polygons, and the $2$-tuple\par

\begin{center}
 $\mathcal{L}'=(|D_c\cap \mathcal{D}|, N_{m})$
\end{center}

is minimized in the lexicographic order.\par

\end{definition}

\begin{remark}
The second tuple $N_{m}$ in the above definition is minimized to reduce similar situations in Figure \ref{f9}(d). In which, it shows such a reduction through an isotopy that is the ``inverse'' of a $\partial$-compression induced by the compression along the green disk.

\end{remark}

\begin{proposition}\label{normal}

Let $S$ be a surface properly embedded in standard position with respect to $\mathcal{D}$, $D_c$ be a compressing disk of $S$. Assume $\mathcal{G}_S$ is nontrivial, then $D_c$ can replaced by a union of disk(s) $\{D^{(l)}_c\}$ , so that each $D^{(l)}_c$ is a compressing disk of $S$ in normal position. \par

\end{proposition}

\begin{figure}[h!]
\centering
  \includegraphics[width=13cm]{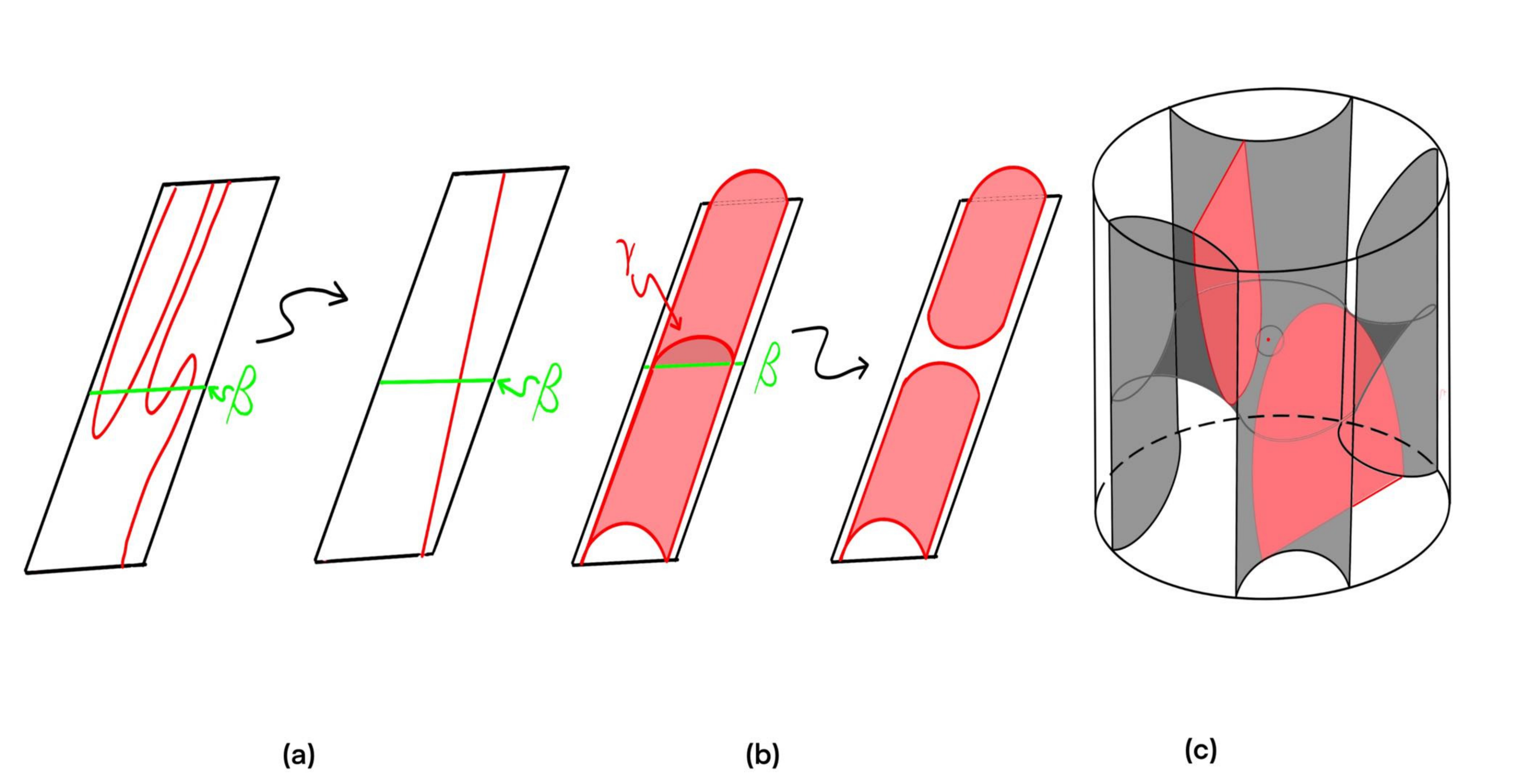}
  \caption{} \label{f8}
\end{figure}

\begin{figure}[h!]
\centering
  \includegraphics[width=16cm]{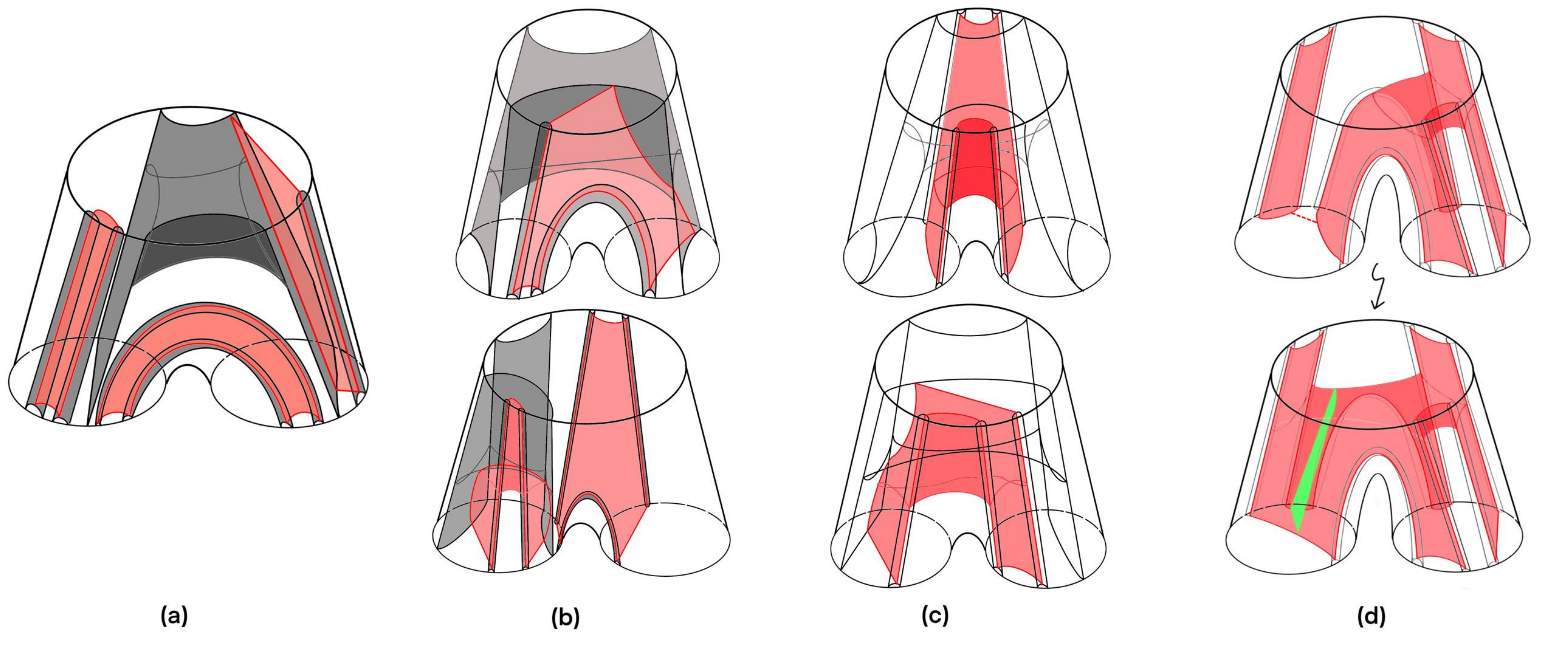}
  \caption{} \label{f9}
\end{figure}

\begin{proof}

Suppose a compressing disk $D_c$ for $S$ is placed so that $|D_c \cap \mathcal{D}|$ is minimized. $D_c \cap \mathcal{D}$ consists of arc(s), since $D_c$ is a disk any loops can be isotoped-off by a normal innermost-loop argument. \par

Let $\beta$ denote the essential horizontal arc intersects with $\partial p_{D_c}$. We assume there exists an edge $\alpha \subset \partial p_{D_c}$ intersecting $\beta$ more than once, then we can deform $\partial p_{D_c}$ on $S$, as shown in Figure \ref{f8}8(a), so that $\alpha$ intersects each horizontal arc no more than once, otherwise $|D_c \cap \mathcal{D}|$ can be reduced. Therefore each edge of $\partial p_{D_c}$ must intersect $\beta$ for at most once. We assume there exists two edges of a polygon $p_{D_c}$ intersecting $\beta$. Then for the horizontal disk $D$ containing $\beta$, we have $D \cap p_{D_c}=\gamma$ is an arc whose boundary points are both contained in $\beta$. Without loss of generality, we assume the disk $D_{\beta \gamma}$ bounded by $\beta$ and $\gamma$ does not intersect $D_c$ in its interior, then we can perform a surgery along $D_{\beta \gamma}$, shown in Figure \ref{f8}8(b). The resultant of this surgery performed on $D_c$ is a disjoint union of two disks, that we denote as $D'_c$ and $D''_c$. At least one of $D'_c$ and $D''_c$ is still a compressing disk, otherwise the original disk $D_c$ would be trivial. We ignore the trivial disk(s) and consider the resultant compressing disk(s) after the surgery, and we assume without loss of generality that $D'_c$ is a compressing disk. Then we can deform $D'_c$ to strictly reduce the intersection number $|S \cap \mathcal{D}|$, until it is minimized.  $D'_c$ still has nonempty intersections with $\mathcal{D}$, otherwise it would contradict the assumption that $\mathcal{G}_S$ is nontrivial. Due to the compactness of $S$ and the disks, we can repeat the above procedure for finite many times on each of the resultant compressing disks and ignore the trivial disks, and then the resultant disks are a union of compressing disk(s) $\{D^{(l')}_c\}$ each satisfying condition 1.\par

Consider each of the polygons $D^{(l')}_c\cap P_k$. We apply the notation $p_{D_c}$ to represent such a polygon for convenience. The construction of $S$ in standard position, and condition 1 implies we can deform the boundary of $p_{D_c}$, so that there is only one critical point $c$ in an edge of $p_{D_c}\cap S$, so that $h(c)$ is either the local maximal or the local minimum. We can also alter $\partial p_{D_c}$ so that $c$ is in a small neighborhood of the critical point in $S_{2n}$, to satisfy condition 2. For the rest of this proof we assume condition 1 and 2 are always satisfied by any compressing disks mentioned below.\par

According to property 5 in Proposition \ref{no saddle}, a $6$-disk  has to intersect three disks in $\mathcal{D}$. And a $4$-disk that contains a critical point has to intersect two disks of $\mathcal{D}$. Therefore, the sub-arc of a bi-gon which contains a local extreme point cannot intersect with a $4$-disk or a $6$-disk. Thus, we have shown that if condition 1 and 2 are guaranteed, then condition 3 has to be satisfied by each compressing disk $D^{(l')}_c$.  See Figure \ref{f8}8(c) for an example of the intersection of two bi-gons with a $12$-disk in some product $W_i$.  \par

It is easy to see the ``if'' part, since similar to the proof of Theorem \ref{essential}, we can isotope $p_{2m}$ in this situation so that it is a non-degenerate saddle disk. See Figure \ref{f9}(c) for two examples of a octagon, and see the lower picture in Figure \ref{f9}(d) for an example of a $16$-gon, in which there exists $6$ edges intersect with disks of different heights in $\mathcal{D}$ alternatingly. To see the ``only if'' part, we suppose there are no such edges which satisfy the assumptions in condition 4. Then if $m\leq 3$, we can place $p_{2m}$ so that it is not a non-degenerate saddle. See Figure \ref{f9}(a), (b). If $m\geq 4$, we can always isotope $p_{2m}$ so that it consists of critical point(s) of $h_{\partial D_c}$ on its edges in  $p_{2m}\cap S$, but it does not consist of any critical points of $h_{D_c}$ in its interior. See the upper picture in Figure \ref{f9}(d) for such an example.
\end{proof}

The above proposition allows us to divide a compressing disk, so that it can be replaced by a union of compressing disks in normal position. Consider $D_c$ a disk embedded in the complement of $S$, so that $|D_c\cap \mathcal{D}|$ is minimized, and $D_c$ intersects each ps-ball $P_k$ in polygons. Each polygon cosists of two type of boundary edges (arcs): The set consisting all arcs of $D_c\cap \mathcal{D}$ will be denoted as $\mathcal{E}_{\mathcal{D}}$. And, the set consisting all arcs of $D_c\cap S-\mathcal{D}$ will be denoted as $\mathcal{E}_S$.  Moreover, we use $n_2$, $n_4$, ... ,$n_{2m}$, where $2m$ is an even integer, to denote the number of bi-gons, rectangles, ..., and $2m$-gons in all $D_c\cap P_k, 1\leq k \leq 2g-2.$ We give an Euler characteristic calculation for the polygons numbers in a compressing disk.\par

Let $\mathcal{P}_{2m}$ denote the set of distinct $2m$-gons, so that in each $\mathcal{P}_I, I=2,4,...,2m$, any two of the polygons are not parallel to each other; Let $\mathcal{P}^J_I$ denote the subset of $\mathcal{P}_I$ including all the $I$-gons that share common edge(s) with some $J$-gon(s). If we can choose elements from $\mathcal{P}^J_I, I,J=2,4,...,2m,$ (possibly identical and parallel to each other) to glue a connected embedded surface $S_D$, so that $\partial S_D$ only consists of edges in $\mathcal{E}_{\mathcal{D}}$, and the number of polygons and their edges in $S_D$ satisfy the a linear equation set, we simply say some elements in $\mathcal{P}^J_I, I,J=2,4,...,2m$, satisfy the equation set.
\par

\begin{lemma} \label{iff}

Let $S$ be a surface properly embedded in standard position with respect to $\mathcal{D}$. Suppose $S$ has no movable saddles, and $\mathcal{G}_S$ in nontrivial. Then $S$ is compressible if and only if some elements in $\mathcal{P}^J_I, I,J=2,4,...,2m,$ satisfy the following equation set:\par

\begin{equation}\label{equation}
    \begin{cases}
      \frac{1}{2}(n_2+(-1)\cdot n_6+...+(2-m)\cdot n_{2m})=1\\
      n_2+(2)\cdot n_4+...+(m)\cdot n_{2m}=|\mathcal{E}_{\mathcal{D}}|=|\mathcal{E}_S|\\
    \end{cases}
\end{equation}

\end{lemma}

\begin{proof}

If some elements in $\mathcal{P}^J_I, I,J=2,4,...,2m,$ satisfy equation set (\ref{equation}), we can find a disk $D_c$ in the complement of $S$, since in a handlebody the only connected, embedded surface with euler characteristic equal to one must be a disk. We claim $D_c$ is a compressing disk of $S$. Because if we assume $D_c$ is a trivial disk, then we can isotope $D_c$ with $\partial D_c$ fixed on $S$ to a disk $D_S\subset S$ bounded by $\partial D_c$. But this cannot be true, because near the local extreme point of such a disk, there has to exist an edge of a polygon that intersects $\mathcal{D}$ twice, which would contradict condition 1 of Proposition \ref{normal}. We notice that the compressing disk $D_c$ glued by such elements in $\mathcal{P}^J_I, I,J=2,4,...,2m,$ satisfying equation set (\ref{equation}) is not necessarily in normal position.\par

By Proposition \ref{normal}, if $S$ is compressible, then there exists a compressing disk $D_c$ in normal position. $D_c$ will be divided by arcs of $D_c \cap \mathcal{D}$ into polygons. We use $\mathcal{V}_{D_c}$, $\mathcal{E}_{D_c}$, and $\mathcal{F}_{D_c}$ to denote the vertex number, edge number, and face number, respectively, in a compressing disk $D_c$. Then since the pull-back graph on $D_c$ of $D_c\cap \mathcal{D}$ is 3-valent, we have $\chi(D_c)= \mathcal{V}_{D_c}-\mathcal{E}_{D_c}+\mathcal{F}_{D_c}=\frac{1}{2}(2n_2+4n_4+6n_6+...+2m\cdot n_{2m})-\frac{3}{2}(n_2+2n_4+3n_6+...+m\cdot n_{2m})+(n_2+n_4+n_6+...+ n_{2m})=\frac{1}{2}(n_2+(-1)\cdot n_6+...+(2-m)\cdot n_{2m})=1$. Half the edge number in the pullback graph of $D_c$ is equal to $\frac{1}{2}(2n_2+4n_4+6n_6+...+2m\cdot n_{2m})=|\mathcal{E}_{\mathcal{D}}|=|\mathcal{E}_S|$. This is because each polygon has half of its edges in the interior of $D_c$, which will intersect $\mathcal{D}$ in $|\mathcal{E}_{\mathcal{D}}|$ arcs; And each polygon has the other half of its edges in $\partial D_c$, which will intersect $S$ in $|\mathcal{E}_S|$ arcs. 
\end{proof}

Let $C_{2x}$ denote the number of $2x$-disks that contains a critical point. We now prove a technical lemma before we state the algorithm mentioned in Theorem \ref{THM1}.\par

\begin{lemma}\label{finiteness}
Let $S$ be a surface properly embedded in standard position with respect to $\mathcal{D}$, so that $S$ consists of no movable saddles. Suppose $D_c$ is a compressing disk in standard position, then the amount of bi-gons contained in $D_c$ which are parallel and considered the same element in $\mathcal{P}_2$ is bounded by $\mathcal{N}_C=\frac{1}{2}\sum_{x=3}^{N}{x\cdot C_{2x}}$, where $2N$ is the largest possible number for the edges of a $2n$-disk.
\end{lemma}

\begin{proof}

We consider two identical elements in $\mathcal{P}_2$, i.e., two bi-gons, $p_2$ and $p'_2$ that are parallel to each other. Assume $p_2$ and $p'_2$ are both sub-disks of a compressing disk $D_c$ in normal position. Then in $D_c$ there must exist a polygon $p_{2m}$ that is path-connected to both $p_2$ and $p'_2$. Moreover, $p_{2m}$ must consists of a critical point, because $p_2$ and $p'_2$ both consist of a local maximum or local minimum point of $D_c$. But this implies there must exist at least one $2N$-disk $D_{2N}$ which consists of a critical point, such that $D_{2N}$ also intersects with two polygons $p_a$ and $p_b$. And each of $p_a$, $p_b$ is path-connected to either $p_2$ or $p'_2$ in $D_c$. Because otherwise, $p_{2m}$ would be a polygon that consists of a critical point, and two edges of $p_{2m}$ both intersect with one $2n$-disk $D_{2N'}$. This leads to a contradiction due to property (1) of Proposition \ref{polygon}. Suppose there are $x_0$ parallel polygons, it would require at least one pair of $2n$-disks with $2x_1,2x_2$ edges, where $x_1,x_2 \geq 4$, $2x_0\leq x_1\leq x_2 \leq N$ (or $x_1=3$, $2=x_0< x_1\leq x_2 \leq N$), so that it would not lead to the similar contradiction mentioned above. Therefore the possible times that an element in $\mathcal{P}_2$ being selected is bounded by $\frac{1}{2}\sum_{x=3}^{N}{x\cdot C_{2x}}$.
\end{proof}
\par

\subsection*{\small{The Algorithm}}\label{Algorithm}

For any pseudo-essential surface $S$, we have the start from step $I$ of the algorithm:\par

$I$ According to Proposition \ref{essential}, $S$ can be isotoped so that each of its disk component is as described in Definition \ref{2ndisk}. Then we proceed to $II$.

$II$. We reducing the complexity $(|S\cap \mathcal{D}|,|f^{-1}(L)|,|f^{-1}(L-1)|,...,|f^{-1}(2)|)$ in lexicographical order through the isotopies described in $\S$\ref{C4}. Then we minimize the complexity $\mathcal{L}$ by reducing movable saddles, so that $S$ is in standard position. According to property 1 of Proposition \ref{no saddle}, if we find a loop of $\mathcal{G}_S$ in some $P_k$, we conclude that $S$ is compressible, and this algorithm finishes; Otherwise, according to the proof of Lemma \ref{L} and the proof of property 2 in Proposition \ref{no saddle}, $S$ can be isotoped so it contains no movable saddle, and we proceed to $III$.

$III$. By Lemma \ref{iff} we know that we can determine the compressibility or incompressibility of $S$ by confirming whether there exists a compressing disk in normal position, through examining whether the elements in $\mathcal{P}^J_I, I,J=2,4,...,2m,$ that satisfy equation set (\ref{equation}). According to property 3 of Proposition \ref{normal}, a bi-gon can only intersects with a $2n$-disk, where $n\geq 4$. Otherwise we can reduce $|S\cap \mathcal{D}|$. 

$III.(a)$ Consider $\mathcal{P}_2$, if it is empty, then apparently we can conclude that there does not exist any compressing disks of $S$. We conclude that $S$ is incompressible, and this algorithm finishes; Otherwise we proceed to $III.(b)$

$III.(b)$ If $\mathcal{P}_2$ is nonempty, we suppose $S$ is compressible. There must be at least a pair of bi-gons in the selection. Each element in $\mathcal{P}_2$ can be selected for at most $\mathcal{N}_c$ times due to Lemma \ref{finiteness}. If two elements in $\mathcal{P}_2$, say $b_1$ and $b_2$, intersect each other, then $b_1$ and $b_2$ cannot be both selected, because the surface glued would not be an embedded surface. Therefore we can examine all the possible selections of bi-gons, and check the polygons connected to them, (and the polygons connected to those mentioned above ...), to see of there exists a selection of polygons which satisfies equation set (\ref{equation}). If the equation set has a solution, then we conclude that $S$ is compressible, otherwise $S$ is incompressible. \par

We notice that the algorithm is finite and it always concludes, since $n_2$ is bounded by $|\mathcal{P}_2|\times \mathcal{N}_c$, and the equation set has at most finitely many positive integer solutions.\par

\begin{proof}[Proof of Corollary \ref{COR1}]
Since $S$ is an incompressible, properly embedded surface in $V$, we can put $S$ in standard position by Proposition \ref{essential}. Moreover, since $S$ can be isotoped so that $S\subset \partial V \times [0,1]$, we have $|S\cap C|=0$.  Assume $S$ is in standard position, and not $\partial$-parallel. Then there must exists a movable saddle in $S$, which contradicts property 2 of Proposition \ref{no saddle}. Therefore, we have proved the first claim. The second claim is obvious from property 1 of Proposition \ref{no saddle}.
\end{proof}

\begin{remark} A compression body $M_C$ is a connected 3-manifold obtained by adding some 1-handles to $F\times I$ with all the gluing disks lying in $F\times 1$, where $F$ is an orientable closed surface (possibly non-connected), and filling in 3-balls to the resulting 3-manifold along all the 2-sphere components in $F\times 0$, if any. The surface $F\times 0$ with all 2-sphere components removed is denoted by $\partial_- M_C$, and $\partial M_C\setminus \partial_-M_C$ is denoted by $\partial_+M_C$. It is clear that $M_C$ is a handlebody if $\partial_-M_C=\emptyset$. As in the handlebody case, we may put the part $F\times I$ of $M_C$ in a position such that at a regular height, the intersection of the height plane and the part $F\times I$ in $M_C$ is either a disk, or an annulus, or two annuli, and by a similar argument, we may generalize Theorem \ref{THM1} to the case of compression bodies. But the statement of a similar theory on compression bodies is tediously long and repetitive, we omit the details of this case. \end{remark}

\section{Some Examples}\label{C6}

We demonstrate the application of our geometric algorithm mentioned in Theorem \ref{THM1} through some examples.\par

\begin{figure}[h!]
            \centering
            \includegraphics[width=13cm]{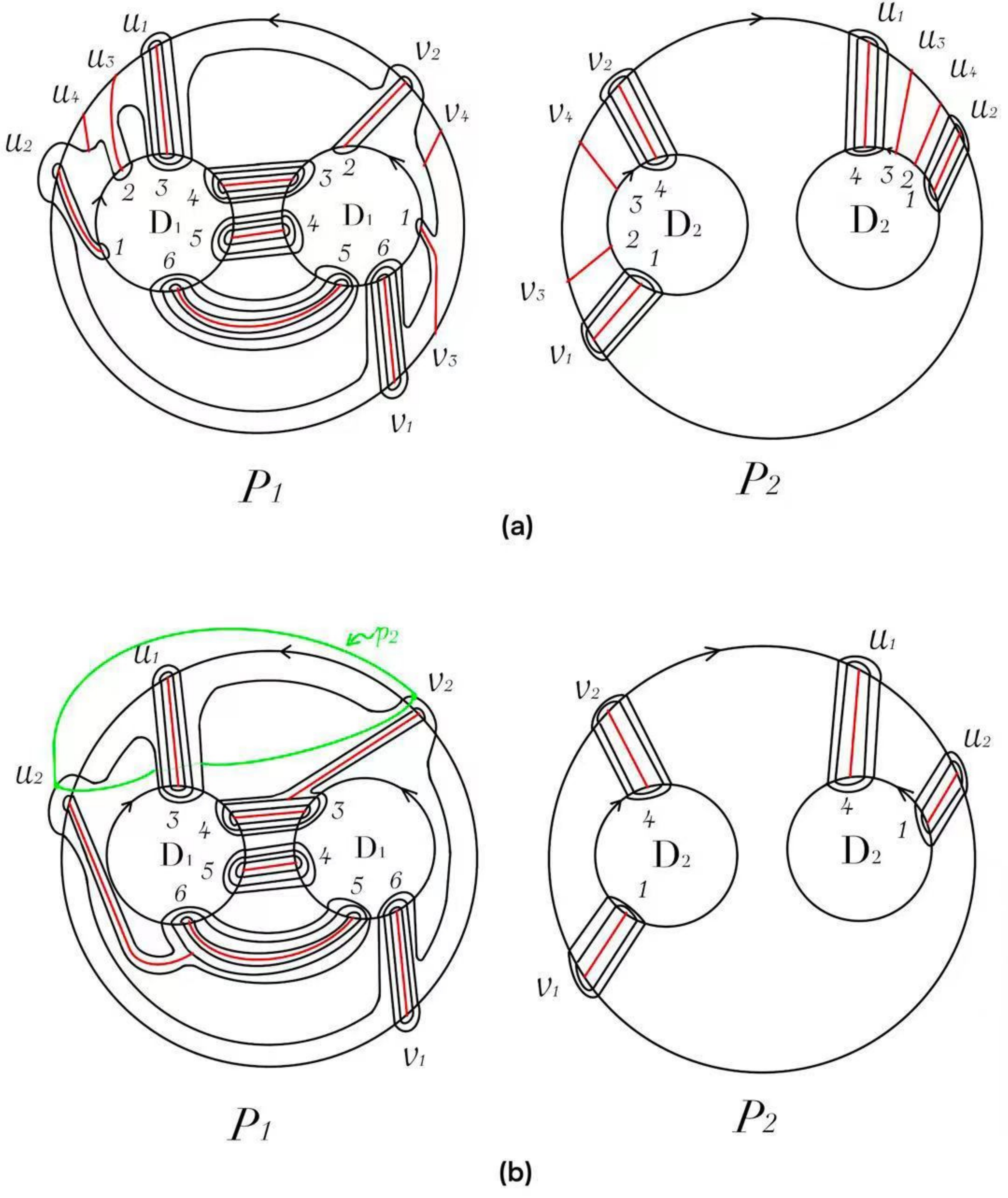}
            \caption{Placing $Q_4$ in standard position.}
            \label{f10}
\end{figure}

\begin{example}
In \cite{8}, Qiu has constructed a series of properly embedded incompressible surfaces $Q_{2n}$, $n\geq1$, of arbitrarily high genus in a genus two handlebody $H_2$. He has proved each $Q_{2n}$ is incompressible by checking that $\pi(Q_{2n})$ is $\pi_1$-injective in $\pi_1(H_2)$. We give a proof using our geometric algorithm to show that each $Q_{2n}$ is incompressible for an arbitrary $n$.
\end{example}

\begin{proof}
Let $D_0$ be a separating disk of $H_2$, and let $(D_1,D_2)$ be a set of meridian disks of $H_2$. Then we take $\mathcal{D}=\{D_0\}\cup \{D_1\}\cup \{D_2\}$, so that $\overline{H_2-\mathcal{D}}=P_1\cup P_2$. It is easy to see that \small{$|Q_{2n}\cap C|$} is at least 1. When $n=1$, $Q_{2n}$ is already in standard position, since any isotopy of $Q_{2n}$ would not decrease the complexity $\mathcal{L}$. Also when $n=1$, it is easy to check that $Q_{2n}$ is incompressible.\par

We are going to picture the surfaces in the style of ``Heegaard diagrams'', in which the red arcs represent unions of $4$-disks, and each closed loop which does not include the parts in the boundaries of $\mathcal{D}$, bounds a disk in $P_1$ or $P_2$. When $n=2$, the graph of $Q_{2n}$ is shown in Figure \ref{f10}(a), and the alignments of $u_i, v_i, 1\leq i \leq 4,$ are as given in \cite{8} . We can isotope $Q_4$  so that it is in standard position as shown in Figure \ref{f10}(b). In $P_1$ of Figure \ref{f10}(b) there exists only one bi-gon $p_2$ in $\mathcal{P}_2$. Therefore, by Lemma \ref{iff} $Q_{2n}$ is incompressible when $n=2$. For $N\geq2$, we can show by ``induction'' that: if $Q_{2N}$ is put in standard position, then we can put $Q_{2N+2}$ in standard position in a similar pattern. Compare $Q_{4}$ with the case of $Q_{6}$ in standard position shown as in Figure \ref{f11} will help us see this. Therefore, we can see for an arbitrary $n$, $Q_{2n}$ always admits only one bi-gon $p_2$ in $P_1$ and $P_2$. Therefore, by Lemma \ref{iff} we have proved $Q_{2n}$ is incompressible for an arbitrary $n$.
\end{proof}

\begin{figure}[h!]
            \centering
            \includegraphics[width=13cm]{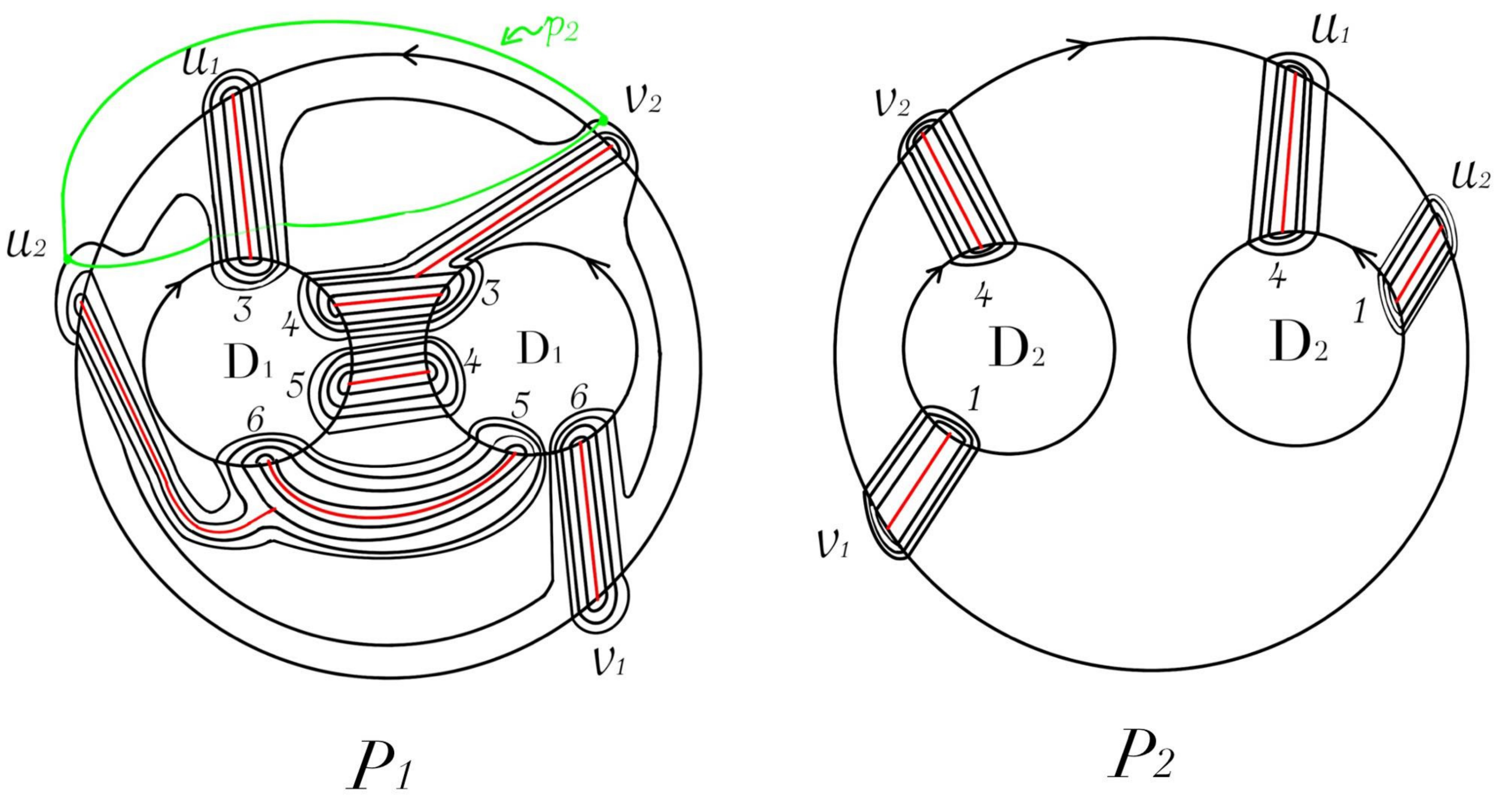}
            \caption{$Q_{6}$ in standard position.}
            \label{f11}
\end{figure}

\begin{figure}[h!]
\centering
  \includegraphics[width=12cm]{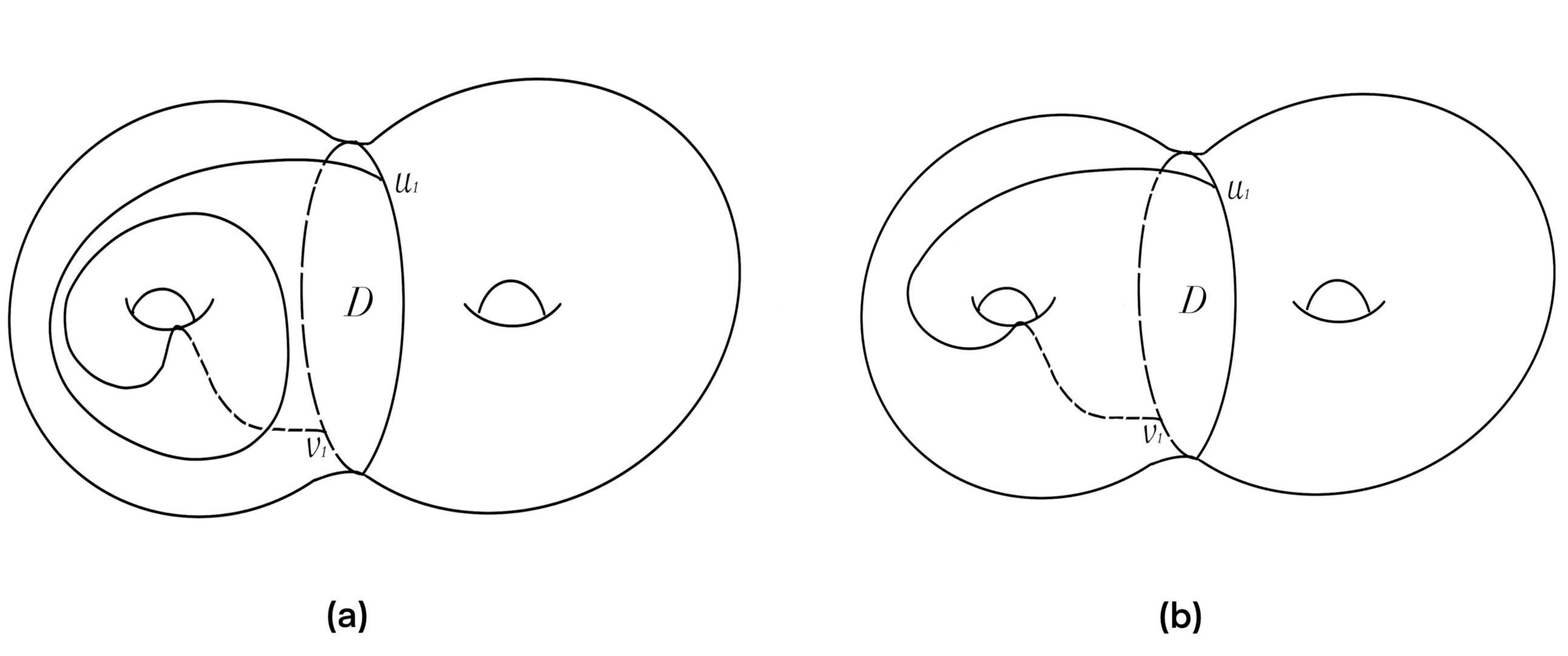}
  \caption{} \label{f12}
\end{figure}

\begin{example}
In \cite{8}, if we replace the arc of $u_1v_1$ by the arc shown in Figure \ref{f12}(a), we can still show that $Q_{2n}$ is incompressible for an arbitrary $n$. If we replace the arc of $u_1v_1$ by the arc shown in Figure \ref{f12}(b), we can show that for $n\geq2$, $Q_{2n}$ is compressible. The proof is similar to that of the preceding example.\par
\end{example}

\begin{example}
In \cite{7}, Jaco constructed a class of non-separating incompressible surfaces $J_n$, $n\geq 1$, of arbitrarily high genus properly embedded in $H_2$. We apply the formal construction of $J_n$ shown in \cite{4}. We take $\mathcal{D}=\{D_0, D_i |i=1,2\}$, and place $J_n$ in standard position. Then the submanifolds are symmetrical with respect to $D_0$ after we cut $H_2$ and $J_n$ with $\mathcal{D}$. $(H_2, J_n)$ can be considered as two ps-ball with subsurfaces of $J_n$ embedded in it, shown as in Figure \ref{f13}, glued together. There exists no bi-gon in each ps-ball for an arbitrary $n$. Therefore, by Lemma \ref{iff} we have proved an arbitrary $n$, $J_n$ is incompressible.

\end{example}

\begin{figure}[h!]
            \centering
            \includegraphics[width=5cm]{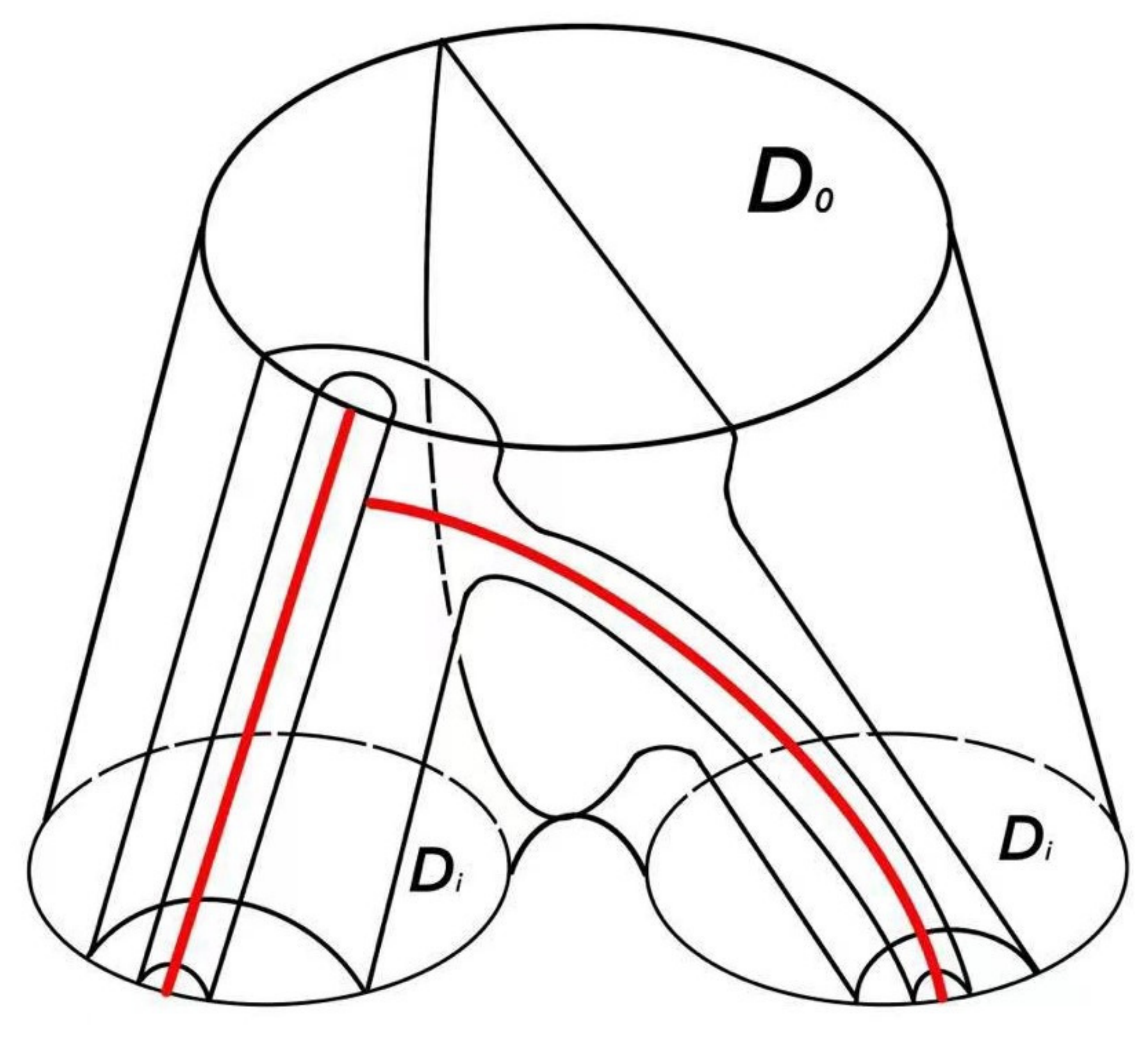}
            \caption{}
            \label{f13}
\end{figure}

%%%%%%%%%%%%%%%%%%%%%%%%%%%%%%%%%%%%%%%%%%%%%%%%%%%%%%%%%%%%%%%%%%%%%%

%%%%%%%%%%%%%%%%%%%%%%%%%%%%%%%%%%%%%%%%%%%%%%%%%%%%%%%
%%% Acknowledgements. ÖÂÐ»
%%%%%%%%%%%%%%%%%%%%%%%%%%%%%%%%%%%%%%%%%%%%%%%%%%%%%%%
{\bf Acknowledgement} The authors would like to express their deep gratitude to the referees for many helpful suggestions.

%%%%%%%%%%%%%%%%%%%%%%%%%%%%%%%%%%%%%%%%%%%%%%%%%%%%%%%
%%% Conflict of interest. ×÷ÕßÀûÒæÉùÃ÷
%%%%%%%%%%%%%%%%%%%%%%%%%%%%%%%%%%%%%%%%%%%%%%%%%%%%%%%
%\InterestConflict

%%%%%%%%%%%%%%%%%%%%%%%%%%%%%%%%%%%%%%%%%%%%%%%%%%%%%%%
%%% Supplements. ²¹³ä²ÄÁÏ, ·Ç±ØÑ¡
%%%%%%%%%%%%%%%%%%%%%%%%%%%%%%%%%%%%%%%%%%%%%%%%%%%%%%%
%\Supplements{}

%%%%%%%%%%%%%%%%%%%%%%%%%%%%%%%%%%%%%%%%%%%%%%%%%%%%%%%
%%% Reference section. ²Î¿ŒÎÄÏ×
%%% citation in the content using "some words~\cite{1,2}".
%%% ~ is needed to make the reference number is on the same line with the word before it.
%%%%%%%%%%%%%%%%%%%%%%%%%%%%%%%%%%%%%%%%%%%%%%%%%%%%%%%

%%%%%%%%%%%%%%%%%%%%%%%%%%%%%%%%%%%%%%%%%%%%%%%%%%%%%%%
%%% Appendix sections. žœÂŒÕÂœÚ, ·Ç±ØÑ¡
%%%%%%%%%%%%%%%%%%%%%%%%%%%%%%%%%%%%%%%%%%%%%%%%%%%%%%%

\end{document}